\documentclass[10pt,reqno]{amsart}
\usepackage{amsmath}
\usepackage{amscd,amsthm,amsfonts,amsopn,amssymb,mathrsfs}
\usepackage{epsfig,hyperref,xcolor}

\newtheorem{theorem}{Theorem}[section]

\newtheorem{proposition}[theorem]{Proposition}

\newtheorem{lemma}[theorem]{Lemma}
\theoremstyle{remark}
\newtheorem{remark}[theorem]{Remark}

\DeclareMathOperator{\iRe}{Re}

\newcommand\conclremark{\leavevmode\unskip\penalty9999\hbox{}\nobreak\hfill\quad\hbox{$\diamondsuit$}}

\allowdisplaybreaks
\begin{document}

\title[Probabilistic NLW on $B_2\times\mathbb{T}$]{Probabilistic well-posedness for the nonlinear wave equation on $B_2\times\mathbb{T}$}
\author[A. Bulut]{Aynur Bulut}
\address{Department of Mathematics, Louisiana State University, 303 Lockett Hall, Baton Rouge, LA, 70803--4918}
\email{aynurbulut@lsu.edu}
\thanks{The author was partially supported by NSF Grant DMS-1361838/DMS-1748083.}

\begin{abstract}
We establish probabilistic well-posedness results for the subcubic nonlinear wave equation, posed on the domain $B_2\times\mathbb{T}$, with randomly chosen initial data having radial symmetry in the $B_2$ variable, and with vanishing Dirichlet boundary conditions on $\partial B_2\times\mathbb{T}$.  
\end{abstract}

\maketitle

\setlength{\parskip}{0.5em}

\section{Introduction}

In this paper, we consider probablistic well-posedness results for the nonlinear wave (NLW) equation on the periodic cylinder $B_2\times \mathbb{T}$, where $B_2\subset\mathbb{R}^2$ denotes the planar unit ball $\{|x'|<1\}$ and where $\mathbb{T}=\mathbb{R}/\mathbb{Z}$ denotes the torus, with radiality of the initial data (and thus of the solution) imposed on the $B_2$ variable.

Our interest in this question is motivated by recent results in the probablilistic well-posedness theory for NLW on general manifolds (for an overview of the classical deterministic well-posedness theory for NLW, see, for instance, \cite{SS}).  In the probabilistic setting, motivated by applications to mathematical physics, study of local and global well-posedness properties of the nonlinear wave and nonlinear Schr\"odinger (NLS) equations, as well as related equations of dispersive type, began with the foundational works \cite{B1,B2,B3} of Bourgain treating NLS on the one, two, and three dimensional torus (see also the work of Lebowitz, Rose and Speer \cite{LRS}).  

Following these works, a large number of authors have studied related issues.  For a comprehensive overview of recent approaches to such results, we refer to the survey article \cite{BenyiOhPocovnicu2019} (see also the earlier survey \cite{B-icm}); in what follows, we point out several particular works which have relevance for our present study.  In \cite{CollianderOh}, Colliander and Oh studied NLS on $\mathbb{T}$ at low regularity, and made use of a probabilistic local theory as part of their arguments.  For NLW on the flat 3D torus $\mathbb{T}^3$, Burq and Tzvetkov have obtained probabilistic global well-posedness results \cite{BT-pwp}, using a variety of energy-based considerations; further results in this direction are due to Pocovnicu \cite{P}, Oh-Pocovnicu \cite{OP}, L\"uhrmann-Mendelson \cite{LM2}, Sun-Xia \cite{SX}, and Oh-Pocovnicu-Tzvetkov \cite{OPN}; see also \cite{Bringmann,BringmannDengNahmodYue}.  We also mention \cite{NPS} for results concerning the Navier-Stokes system. 

Turning to the case of more general compact manifolds, Burq and Tzvetkov established a probabilistic local well-posedness theory for cubic NLW on arbitrary compact three-dimensional manifolds, and moreover have obtained almost sure global well-posedness for NLS and NLW \cite{T-disc,T-inv,BT1,BT2} on the 2D and 3D unit balls $B_2$ and $B_3$, making use of the invariance of the Gibbs measure.  In a recent series of works, the author and J. Bourgain have improved on these results, obtaining almost sure global well-posedness for Gibbs measure evolutions for NLW on the 3D unit ball with arbitrary energy-subcritical nonlinearity in \cite{BB4}, for NLS on the 2D unit ball with higher power nonlinearities in \cite{BB2}, and for cubic NLS on the 3D unit ball in \cite{BB3}; see also \cite{BB1} for an overview of these results.  Recent developments in the two-dimensional general manifold setting include work of Oh-Robert-Tzvetkov \cite{OhRobertTzvetkov}.

In the present work, we consider the special case of the three-dimensional manifold $B_2\times\mathbb{T}$ with radial symmetry on $B_2$, which we view as a case which interpolates between the $B_3$ and $\mathbb{T}^3$ settings (while it scales as a three-dimensional problem from a PDE point of view, the symmetry in the $B_2$ variable causes the random data treatment to be in close analogy to the spectral theory associated to two-dimensional problems).  At a technical level, this interpolation results from the difference in eigenvalue asymptotics between $B_2\times\mathbb{T}$ and $B_3$, along with the presence of $L_x^\infty$ eigenfunction estimates in the periodic variable.

We begin with some preliminary notation.  Write $x=(x',x_3)\in B_2\times\mathbb{T}$ and let $(e_n)$ be the sequence of radial eigenfunctions of $-\Delta$ on $B_2$ (with vanishing Dirichlet boundary conditions), ordered so that the associated eigenvalues $\lambda_n^2$ are increasing.  Introducing the notation
\begin{align}
z_{n,n'}:=\sqrt{\lambda_n^2+(2\pi n')^2}.\label{ref1}
\end{align}
so that $\{z_{n,n'}^2:n,n'\in\mathbb{Z}, n\geq 1\}$ are the eigenvalues of the operator $-\Delta$ on $B_2\times\mathbb{T}$ (restricted to functions which are radial in the $x'$ variable), we observe that the space of eigenfunctions for this operator is generated by the family of maps
\begin{align*}
(x',x_3)\mapsto e_n(x')e^{2\pi in'x_3},\,\, n\in\mathbb{N}, n'\in\mathbb{Z}.
\end{align*}

Let $H_x^s(B_2\times \mathbb{T})$, $s\in\mathbb{R}$, denote the space of functions $f:B_2\times\mathbb{T}\rightarrow \mathbb{C}$ with
\begin{align*}
f(x)=\sum_{n\geq 1}\sum_{n'\in\mathbb{Z}} a_{n,n'}e_n(x_1,x_2)e^{2\pi i n'x_3},
\end{align*}
for some $(a_{n,n'})_{n\geq 1,n'\in\mathbb{Z}}$ such that the associated norm
\begin{align}
\lVert f\rVert_{H_x^s}:=\bigg(\sum_{n\geq 1}\sum_{n'\in\mathbb{Z}} \langle z_{n,n'}\rangle^{2s}|a_{n,n'}|^2\bigg)^{1/2}\label{ref4}
\end{align}
is finite. 

Fix $\alpha\in\mathbb{R}$ and a pair $(F,G)\in H_x^{\alpha}\times H_x^{\alpha-1}$ of real-valued functions whose Fourier series representations are given by 
\begin{align}
(F,G)=\bigg(\sum_{n,n'} \alpha_{n,n'}e_n(x_1,x_2)e^{2\pi i n'x_3},\sum_{n,n'} \beta_{n,n'}e_n(x_1,x_2)e^{2\pi i n'x_3}\bigg)\label{eq-fg}
\end{align}
where $\alpha_{n,n'}=\overline{\alpha_{n,-n'}}$ and $\beta_{n,n'}=\overline{\beta_{n,-n'}}$ for each pair $(n,n')$ (as a consequence of the fact that $F$ and $G$ are real-valued).  Letting $(\Omega,\mathscr{M},p)$ be a given probability space, we shall establish a local well-posedness result for the nonlinear wave equation with power-type nonlinearity,
\begin{align*}
\textrm{(NLW)}\quad \left\lbrace\begin{array}{ll}w_{tt}-\Delta w+\lambda |w|^\gamma w=0,&\textrm{on}\,\, I\times (B_2\times \mathbb{T}),\\
(w,w_t)(0)=(F_\omega,G_\omega),&\textrm{on}\,\, B_2\times\mathbb{T},\\
w(t,x_3)|_{\partial B_2}=0,&t\in I, x_3\in \mathbb{T},
\end{array}\right.
\end{align*}
for $\lambda\in\{-1,1\}$, with random initial data $(F_\omega, G_\omega)$ given by 
\begin{align}
\label{ref2}F_\omega=\sum_{n,n'} h_{n,n'}(\omega)\alpha_{n,n'}e_n(x_1,x_2)e^{2\pi i n'x_3},
\end{align}
and
\begin{align}
\label{ref3}G_\omega=\sum_{n,n'} k_{n,n'}(\omega)\beta_{n,n'}e_n(x_1,x_2)e^{2\pi i n'x_3},
\end{align}
for $\omega\in \Omega$, where $(h_{n,n'})$ and $(k_{n,n'})$ are taken as sequences of independent standard complex-valued Gaussian random variables on $\Omega$ satisfying $h_{n,n'}=\overline{h_{n,-n'}}$ and $k_{n,n'}=\overline{k_{n,-n'}}$.  

The significance of this choice of randomization lies in the specification of the regularity of the randomized data.  For each given $\alpha\in\mathbb{R}$, the above randomized initial data $(F_\omega,G_\omega)$ now belongs almost surely to $H_x^\alpha(B_2\times\mathbb{T})\times H_x^{\alpha-1}(B_2\times\mathbb{T})$; moreover, for $\epsilon>0$, if $(F,G)\not\in H_x^{\alpha+\epsilon}(B_2\times \mathbb{T})\times H_x^{\alpha-1+\epsilon}(B_2\times\mathbb{T})$ then $(F_\omega,G_\omega)$ does not belong to $H_x^{\alpha+\epsilon}\times H_x^{\alpha-1+\epsilon}$ almost surely (see, e.g. Lemma $B.1$ in \cite{BT1}).

To establish local well-posedness, we use a reformulated form of the equation (NLW), arising from the substitution $u=w+i(\sqrt{-\Delta})^{-1}w_t$ (where $w$ is a solution to (NLW)), which takes the form
\begin{align}
\left\lbrace\begin{array}{rl}iu_t-\sqrt{-\Delta}u-\lambda(\sqrt{-\Delta})^{-1}\Big(|\iRe u|^\gamma\iRe u\Big)&=0,\\
u|_{t=0}&=\phi_\omega,\end{array}\right.\label{ref5}
\end{align}
with
\begin{align}
\phi_\omega=F_\omega+i(\sqrt{-\Delta})^{-1}G_\omega=\sum_{n\geq 1}\sum_{n'\in\mathbb{Z}} g_{n,n'}(\omega)\gamma_{n,n'}e_n(x_1,x_2)e^{2\pi i n'x_3},\label{ref6}
\end{align}
where $(\gamma_{n,n'})$ and $(g_{n,n'})$ are chosen so that $g_{n,n'}\gamma_{n,n'}=\alpha_{n,n'}h_{n,n'}+i\frac{\beta_{n,n'}k_{n,n'}}{z_{n,n'}}$ for all $n\geq 1$ and $n'\in\mathbb{Z}$, with $(g_{n,n'})$ a sequence of independent standard complex-valued Gaussian random variables on $(\Omega,\mathscr{M},p)$ (e.g. for each pair $(n,n')$, one can set $\gamma_{n,n'}:=|\alpha_{n,n'}|+i\frac{|\beta_{n,n'}|}{z_{n,n'}}$ and, when $\gamma_{n,n'}\neq 0$, $g_{n,n'}(\omega):=\tfrac{\alpha_{n,n'}}{\gamma_{n,n'}}h_{n,n'}(\omega)+i\tfrac{\beta_{n,n'}/z_{n,n'}}{\gamma_{n,n'}}k_{n,n'}(\omega)$, while when $\gamma_{n,n'}=0$, $g_{n,n'}(\omega):=\tfrac{1}{\sqrt{2}}h_{n,n'}(\omega)+\tfrac{1}{\sqrt{2}}k_{n,n'}(\omega)$; with these choices, the sequence $(g_{n,n'})$ is indeed a sequence of independent standard complex-valued Gaussians).

Denoting the evolution operator associated to the linear equation $iu_t-(\sqrt{-\Delta})u=0$ by
\begin{align}
S(t)\phi&=\sum_{n,n'} a_ne_n(x_1,x_2)e^{i(2\pi n'x_3-z_{n,n'}t)}\label{ref7}
\end{align}
whenever $\phi:B_2\times\mathbb{T}\rightarrow\mathbb{C}$ is given by
\begin{align}
\phi(x)=\sum_{n,n'} a_ne_n(x_1,x_2)e^{2\pi in'x_3},\label{ref8}
\end{align}
we consider solutions to ($\ref{ref5}$) in the sense of the Duhamel formula
\begin{align}
u(t)=S(t)\phi_\omega-i\lambda\int_0^t S(t-t')(\sqrt{-\Delta})^{-1}\Big[|\iRe u|^\gamma \iRe u(t')\Big]dt'.\label{ref9}
\end{align}

Our main theorem is a local well-posedness result for the initial value problem (NLW), in the form of the following theorem.  

\begin{theorem}[Local well-posedness for (NLW)]
\label{ref10}
Fix $0<\gamma<2$ and let $\alpha\in\mathbb{R}$ be such that
\begin{align}
\alpha&>0\quad\textrm{if}\quad \gamma<\sqrt{\tfrac{5}{3}},\quad\textrm{and}\label{thm1cond1}\\
\alpha&>\frac{\gamma-1}{2(\gamma+1)}\quad\textrm{if}\quad \sqrt{\tfrac{5}{3}}\leq \gamma<2.\label{thm1cond2}
\end{align}

Let $(F,G)\in H_x^\alpha(B_2\times\mathbb{T})\times H_x^{\alpha-1}(B_2\times\mathbb{T})$ be a pair of real-valued functions whose Fourier series representations are as in \eqref{eq-fg}.  For $\omega\in \Omega$, let $(F_\omega,G_\omega)$ be the randomized pair defined in \eqref{ref2}--\eqref{ref3}.  Moreover, for $(n,n')$ and $\omega\in \Omega$, let $\gamma_{n,n'}$, $g_{n,n'}(\omega)$, and $\phi_\omega$ be as defined in \eqref{ref6}.

Then for every $A\geq 1$ there exists a set $\Sigma_A\subset\Omega$ with $p(\Sigma_A)\leq C_1\exp(-C_2A^{2})$ such 
that for all $\omega\in\Omega\setminus\Sigma_A$ the initial value problem \eqref{ref5} has a unique solution $u$ on the 
interval $I=[0,A^{-c}]$, with 
\begin{align}
u-S(t)\phi_\omega\in X^{s,b}(I)\label{ref12}
\end{align}
for some $s\in [\frac{1}{2},1]$ (depending on $\gamma$) and $b>\frac{1}{2}$.
\end{theorem}

The proof of Theorem $\ref{ref10}$ is based on a fixed-point analysis in $X^{s,b}$ spaces; see Section $2$ for the definition of these spaces in our context.  We outline three key ingredients in the argument, each of which makes essential use of the product structure of the domain:
\begin{itemize}
\item a class of local-in-time Strichartz estimates adapted to our $B_2\times\mathbb{T}$ setting, making use of eigenfunction expansions of the type given in ($\ref{ref8}$); see Proposition $\ref{ref21}$ --- it is important to note that the present setting is made delicate by the presence of ``essentially repeating'' eigenvalues (that is, $(n_1,n'_1)\neq (n_2,n'_2)$ with $z_{n_1,n'_1}\sim z_{n_2,n'_2}$) --- to account for this, we make use of the technique of spectral projectors (this is described further in Section $2.2$ below; see also \cite{BLP}),
\item a class of large-deviation type estimates for the linear evolution, adapted to the product structure of the domain and making full use of the eigenvalue asymptotics in our setting; see Section $3$, and
\item an estimate of the nonlinearity which allows us to control the $X^{s,b}$ norm of the nonlinear term in the Duhamel formula by a suitable $L_x^pL_t^2$ norm; see Lemma $\ref{ref44}$.
\end{itemize}

Recall that initial data of the form ($\ref{ref6}$) is, almost surely in $\omega$, in the regularity class $H_x^{\alpha}$.  When $\alpha$ and $\gamma$ satisfy $\alpha<\frac{3}{2}-\frac{2}{\gamma}$ the problem therefore belongs to the ill-posed (i.e. supercritical) regime (see, e.g. \cite{BT1} for a treatment of the cubic case), and probabilistic considerations are essential.  We remark that Theorem $\ref{ref10}$ does not cover the case of the cubic nonlinearity $\gamma=2$ (in this context, the restrictions on the nonlinearity given by ($\ref{thm1cond1}$)--($\ref{thm1cond2}$) correspond to $\alpha>\frac{1}{6}-$ for $\gamma=2-$).  These restrictions arise from technical aspects of our estimates (see, e.g. Lemma \ref{ref46}), and are likely not sharp.  

For simplicity of our presentation, we do not address the extension of this local solution to a global one in the present paper.  While techniques based on Gibbs measure considerations do not readily apply in our context (as the associated initial data has low regularity, and the nonlinearity would therefore need to be renormalized in order to be interpreted properly), a variety of other tools from the deterministic and probablilistic theory readily apply in our context.  This includes methods based on high-low frequency decompositions (c.f. \cite{CollianderOh} and \cite{LM}), as well as methods based on energy considerations (see, for instance \cite{BT-pwp,P,OP,LM2}).
 
We conclude this introduction by giving an outline of the rest of this paper.  In Section 2, we establish some notational conventions which will be useful in the rest of the article, including the specification of the function spaces of $X^{s,b}$ type which will be used for our local-wellposedness arguments.  We also obtain the relevant Strichartz estimates.  In Section 3 we obtain the relevant probabilistic estimates, which are bounds of large deviation type, and which rely heavily on the product structure of the domain $B_2\times\mathbb{T}$ in order to obtain the optimal amount of integrability.  Section $4$ is then devoted to estimates of the nonlinearity, which are applied in Section $5$ to obtain the proof of Theorem $\ref{ref10}$.  In Appendix A, we collect some elementary estimates which show how the condition ($\ref{thm1cond1}$)--($\ref{thm1cond2}$) enables the choice of regularity and integrability parameters used in the proof of Theorem $\ref{ref10}$.

\section{Preliminaries and function spaces}

In this section, we establish some notation and specify the main function spaces (of $X^{s,b}$ type) which will underlie our arguments in the remainder of the paper.  

We will frequently write $\langle x\rangle=(1+|x|^2)^{1/2}$, and make use of the notation $z_{n,n'}$ defined in ($\ref{ref1}$) for $n\in\mathbb{N}$ and $n'\in \mathbb{Z}$.  In this context, summations over $n$ will typically be taken over $\mathbb{N}$, while summations over $n'$ will be taken over $\mathbb{Z}$.  As usual, we will interchangably use the notations $A\lesssim B$ and $A\leq CB$, $C>0$, and the value of $C$ may change from line to line (unless otherwise indicated).  Moreover, $I\subset\mathbb{R}$ will be used to denote an arbitrary time interval $I=[0,T]$ with $|I|\leq 1$.

We will also use the notation $\mu$ to denote the probability measure induced by the Gaussian process
\begin{align*}
\omega\mapsto \phi_\omega=\sum_{n\geq 1}\sum_{n'\in\mathbb{Z}} g_{n,n'}(\omega)\gamma_{n,n'}e_n(x_1,x_2)e^{2\pi i n'x_3},
\end{align*}
i.e. the measure given by
\begin{align*}
\mu:A\mapsto p(\{\omega:\phi_\omega\in A\}).
\end{align*}

We now recall some basic properties of the eigenfunctions and eigenvalues of $-\Delta$ on $B_2$ which will play a fundamental role in our analysis below.\footnote{It is illuminating to compare ($\ref{ref16}$)--($\ref{ref17}$) with the $B_3$ estimates used in \cite{BB4,BT2}.  The difference in homogeneity, combined with the presence of $L_x^\infty$ estimates on the eigenfunctions $e^{2\pi i n'x_3}$ in the $\mathbb{T}$ variable (such estimates are indeed classically relevant; see, e.g. \cite{B-icm}), in some sense corresponds to the intuition that our results interpolate between the $B_3$ and $\mathbb{T}^3$ cases.}
Recall that $(e_n)$ and $(\lambda_n^2)$ denote the sequences of radial eigenfunctions and associated eigenvalues of $-\Delta$ on $B_2$ (with vanishing boundary conditions, and arranged so that $\lambda_n<\lambda_{n+1}$ for all $n$).  It follows from standard estimates on Bessel functions (see, e.g. \cite{T-disc}) that one has the asymptotic bound
\begin{align}
\Big|\lambda_n-\big(n\pi-\tfrac{\pi}{4}\big)\Big|\leq Cn^{-1}.\label{ref16}
\end{align}
Moreover, similar arguments 
yield the following $L^p$ estimates for the eigenfunctions $e_n$:
\begin{align}
\lVert e_n\rVert_{L_x^p(B_2)}&\lesssim \left\lbrace\begin{array}{ll}1,&1\leq p<4,\\\log(2+n)^{1/4},&p=4,\\n^{\frac{1}{2}-\frac{2}{p}},&p>4.\end{array}\right.\label{ref17}
\end{align}

\subsection{Function spaces}

We now specify the function spaces which will be used in the rest of the paper.  For each $s\in\mathbb{R}$, we shall use $H_x^s(B_2\times \mathbb{T})$ to denote the usual Sobolev space given by the norm ($\ref{ref4}$).  As described in the introduction, we shall also make use of $X^{s,b}$ spaces adapted to our context.  In particular, fixing $s\in\mathbb{R}$, $b>\frac{1}{2}$, and a time interval $I=[0,T]\subset \mathbb{R}$ with $0<T<1$, the space $X^{s,b}(I)$ will denote the space of all functions having representations
\begin{align}
f(t,x)=\sum_{m,n,n'} f_{m,n,n'}e_n(x_1,x_2)e^{2\pi i (n'x_3+mt)},\quad t\in I, x\in B_2\times\mathbb{T},\label{ref18}
\end{align}
such that the norm
\begin{align*}
\lVert f\rVert_{X^{s,b}(I)}&:=\inf_{(f_{m,n,n'})\subset\mathbb{C}}\,\, \left(\sum_{m,n,n'} \langle 2\pi m+z_{n,n'}\rangle^{2b}\langle z_{n,n'}\rangle^{2s}|f_{m,n,n'}|^2\right)^{1/2}
\end{align*}
is finite, where the infimum is taken over all sequences $(f_{m,n,n'})$ leading to the representation ($\ref{ref18}$) on $I\times (B_2\times \mathbb{T})$.

With this notation, we note that for all $s\in \mathbb{R}$, $b>\frac{1}{2}$, one has the (continuous) embedding
\begin{align}
X^{s,b}\hookrightarrow C_t(I;H_x^{s}).\label{ref19}
\end{align}
An additional embedding property is established in Section $\ref{ref43}$ as Lemma $\ref{ref46}$.

\subsection{Strichartz estimates}

We conclude this section by establishing a suitable form of the linear Strichartz estimates associated with the operator $S(t)$ defined in ($\ref{ref7}$).  

\begin{proposition}\label{ref21}
For every $p$, $q$, and $s>0$ satisfying $2\leq p,q<\infty$ and
\begin{align*}
s>\max\bigg\{1-\frac{1}{q},\frac{3}{2}-\frac{1}{q}-\frac{2}{p}\bigg\},
\end{align*}
one has
\begin{align*}
\lVert S(t)f\rVert_{L_x^pL_t^q(I)}&\lesssim \lVert f\rVert_{H_x^s}
\end{align*}
for every $f\in H_x^s(B_2\times\mathbb{T})$ and $I\subset\mathbb{R}$ with $|I|\leq 1$.
\end{proposition}

In order to proceed with the proof of Proposition $\ref{ref21}$, we will need to introduce a spectral projection operator.  In particular, for 
\begin{align*}
f(x)=\sum_{n,n'}a_{n,n'}e_n(x_1,x_2)e^{2\pi i n'x_3},
\end{align*}
let $Af:B_2\times\mathbb{T}\rightarrow\mathbb{C}$ be defined by
\begin{align*}
Af(x)&:=2\pi\sum_{n,n'} a_{n,n'}\lfloor z_{n,n'}/(2\pi)\rfloor e_n(x_1,x_2)e^{2\pi i n'x_3}, \quad x\in B_2\times\mathbb{T},
\end{align*}
with $\lfloor t\rfloor$ denoting the greatest integer less than or equal to $t$ for $t\geq 0$, and note that the operator $A-\sqrt{-\Delta}$ is then bounded on $H_x^s$ for any $s\in\mathbb{R}$.  

We will also make use of the associated evolution $S_A(t)f$ for $t>0$ given by
\begin{align*}
(S_A(t)f)(x)&:=\sum_{n,n'} a_{n,n'}e_n(x_1,x_2)e^{2\pi i(n'x_3-\lfloor z_{n,n'}/(2\pi)\rfloor t)}.
\end{align*}

To establish Proposition $\ref{ref21}$, we note that after setting $u=S(t)f$, the equality
\begin{align*}
iu_t-Au&=-(A-\sqrt{-\Delta})u
\end{align*}
can be rewritten in integral form as
\begin{align}
S(t)f&=S_A(t)f+i\int_0^t S_A(t-t')[(A-\sqrt{-\Delta})S(t')f]dt'.\label{ref22}
\end{align}
It therefore suffices to establish the following lemma (see, e.g. the last step in the proof of Theorem $2.1$ in \cite{BLP}, as well as the references cited there):

\begin{lemma}
\label{ref23}
For $p$, $q$ and $s$ be as stated in Proposition $\ref{ref21}$, one has
\begin{align*}
\lVert S_A(t)f\rVert_{L_x^pL_t^q(I)}&\lesssim \lVert f\rVert_{H_x^s}
\end{align*}
for every $f\in H_x^s(B_2\times\mathbb{T})$ all $0<T<1$, with $I=[0,T]\subset\mathbb{R}$.
\end{lemma}

\begin{proof}
Writing
\begin{align*}
f(x)=\sum_{n,n'} a_{n,n'}e_n(x_1,x_2)e^{2\pi i n'x_3}
\end{align*}
and using the Sobolev inequality in time, we obtain, with $H_t^{\frac{1}{2}-\frac{1}{q}}([0,1])$ denoting the usual fractional Sobolev space on the time interval $[0,1]$,
\begin{align}
\nonumber &\lVert S_A(t)f\rVert_{L_x^pL_t^q(I)}\\
\nonumber &\hspace{0.2in}\lesssim \bigg\lVert \sum_{n,n'} a_{n,n'}e_n(x_1,x_2)e^{2\pi i(n'x_3+\lfloor z_{n,n'}/(2\pi)\rfloor t)}\bigg\rVert_{L_x^pH_t^{\frac{1}{2}-\frac{1}{q}}([0,1])}\\
\nonumber &\hspace{0.2in}=\bigg\lVert \bigg(\sum_k \langle k\rangle^{1-\frac{2}{q}}\bigg|\sum_{\substack{n,n'\\\lfloor z_{n,n'}/(2\pi)\rfloor=k}} a_{n,n'}e_n(x_1,x_2)e^{2\pi i n'x_3}\bigg|^2\bigg)^{1/2}\bigg\rVert_{L_x^p}\\
&\hspace{0.2in}\leq \bigg\lVert \bigg(\sum_k \langle k\rangle^{1-\frac{2}{q}}N(k)\bigg(\sum_{\substack{n,n'\\\lfloor z_{n,n'}/(2\pi)\rfloor=k}} |a_{n,n'}|^2|e_n(x_1,x_2)|^2\bigg) \bigg)^{1/2}\bigg\rVert_{L_x^p},\label{ref24}
\end{align}
where we have set
\begin{align*}
N(k):=\#\{(n,n'):\lfloor z_{n,n'}/(2\pi)\rfloor=k\}.
\end{align*}

It now follows from standard estimates that one has the bound
\begin{align*}
N(k)\leq Ck\,\, \textrm{for all}\,\, k\geq 1
\end{align*}
for a suitable constant $C>0$.  Indeed, we argue following a suggestion of Z. Rudnik \cite{Rudnik}, and note that for fixed $k$, the condition $\{(n,n')\in \lfloor z_{n,n'}/(2\pi)\rfloor=k\}$ corresponds to $$k^2\leq \frac{\lambda_n^2}{(2\pi)^2}+(n')^2< (k+1)^2,$$ so that the asymptotics ($\ref{ref16}$) for $\lambda_n$ give $$k^2\leq \frac{1}{4}(n-\frac{1}{4}+O(\frac{1}{n}))^2+(n')^2< k^2+2k+1,$$ which we can rewrite as $$4k^2\leq n^2-\frac{1}{2}n+4(n')^2+O(1)<4k^2+8k+O(1).$$  Muliplying these inequalities through by $16$ and setting $K:=8k$, this becomes $$K^2\leq (4n-1)^2+(8n')^2+O(1)\leq K^2+16K.$$  It follows that $N(k)$ is essentially controlled by the number of lattice points in an annulus  of inner radius $K$ and outer radius $K+8$.  Since the number of lattice points in a disk of radius $R$ is $\pi R^2+O(R^\theta)$ for some $\theta<2/3$, we can estimate this by $\pi (K+8)^2-\pi K^2+O(K^\theta)=16K\pi+O(K^\theta)=128k\pi+O(k^\theta)=O(k)$ as desired.

We therefore obtain
\begin{align}
\nonumber (\ref{ref24})&\lesssim \bigg\lVert \bigg(\sum_k \langle k\rangle^{2-\frac{2}{q}}\bigg(\sum_{\substack{n,n'\\\lfloor z_{n,n'}/(2\pi)\rfloor=k}} |a_{n,n'}|^2|e_n(x_1,x_2)|^2\bigg) \bigg)^{1/2}\bigg\rVert_{L_x^p}\\
&\leq \bigg( \sum_k\sum_{\substack{n,n'\\\lfloor z_{n,n'}/(2\pi)\rfloor=k}}\langle k\rangle^{2-\frac{2}{q}}|a_{n,n'}|^2\lVert e_n(x_1,x_2)\rVert_{L_x^{p}}^2\bigg)^{1/2}.\label{ref25}
\end{align}

To conclude the proof of the lemma, we fix $\epsilon>0$ and observe from (\ref{ref17}) that for $p\leq 4$, ($\ref{ref25}$) is bounded by
\begin{align}
\bigg(\sum_k\sum_{\substack{n,n'\\\lfloor z_{n,n'}/(2\pi)\rfloor=k}} \langle k\rangle^{2-\frac{2}{q}+2\epsilon}|a_{n,n'}|^2\bigg)^{1/2}\lesssim \lVert f\rVert_{H_x^{1-\frac{1}{q}+\epsilon}}\label{ref26}
\end{align}
while for $p>4$, one gets the bound
\begin{align}
\bigg(\sum_k\sum_{\substack{n,n'\\\lfloor z_{n,n'}/(2\pi)\rfloor=k}} \langle k\rangle^{3-\frac{2}{q}-\frac{4}{p}}|a_{n,n'}|^2\bigg)^{1/2}\lesssim \lVert f\rVert_{H_x^{\frac{3}{2}-\frac{1}{q}-\frac{2}{p}}}.\label{ref27}
\end{align}

Combining ($\ref{ref25}$) with ($\ref{ref26}$)--($\ref{ref27}$) completes the proof of Lemma $\ref{ref23}$.
\end{proof}

\section{Probabilistic estimates of the linear evolution}
\label{ref28}

We now establish a collection of probabilistic estimates on the linear evolution.  These estimates comprise the main probabilistic component of our argument.

For $\alpha>0$ and $0\leq s<\alpha$, define
\begin{align}
\rho_*(\alpha,s)=\left\lbrace\begin{array}{ll}\frac{4}{1-2(\alpha-s)},&\textrm{if}\,\,\, 0<\alpha-s<\frac{1}{2},\\
\infty,&\textrm{if}\,\,\, \alpha-s\geq \frac{1}{2}.\end{array}\right.\label{ref29}
\end{align}
and
\begin{align}
\rho_*(\alpha)=\rho_*(\alpha,0)\label{ref30}
\end{align}

Our probabilistic estimates now take the form of the following lemma:

\begin{lemma}
\label{ref31}
Fix $T\in (0,1)$ and $\alpha>0$.  Then there exists $c>0$ such that the estimates
\begin{itemize}
\item[(i)] for all $0<s<\alpha$ and $1\leq p\leq \rho_*(\alpha,s)$,
\begin{align}
\mu(\{\phi:\lVert (\sqrt{-\Delta})^s\phi\rVert_{L^p}>\lambda\})\lesssim e^{-c\lambda^2},\label{ref32}
\end{align}
and
\item[(ii)] for all $(p,q)\in [1,\infty)^2$ with $1\leq p\leq\rho_*(\alpha)$ and $2\leq q<\infty$,
\begin{align}
\mu(\{\phi:\lVert S(t)\phi\rVert_{L_x^pL_t^q(I)}>\lambda\})\lesssim e^{-c\lambda^2},\label{ref33}
\end{align}
\end{itemize}
are valid for all $\lambda>0$, where $\phi$ is randomly chosen initial data of the form $(\ref{ref6})$.
\end{lemma}

It is important to note that the estimates of Lemma $\ref{ref31}$ make essential use of the product structure of the domain $B_2\times\mathbb{T}$.  Indeed, the reader may find it useful to compare the range $p<\rho_*(\alpha)$ with the corresponding range in the $B_3$ case \cite{BB4}.  In our setting, the gain in integrability is essentially a consequence of the $B_2$ eigenfunction bounds ($\ref{ref17}$) and the fact that eigenfunctions on $\mathbb{T}$ (and indeed $\mathbb{T}^d$ for $d\geq 1$) are bounded in $L_x^\infty$.

\begin{proof}[Proof of Lemma $\ref{ref31}$]
Observing that we have $\rho_*(\alpha,s)>4$, we note that for each of the statements (i) and (ii) there is no loss of generality in considering the case $p>4$.

Indeed, for $p\leq 4$, $\lambda>0$ and $\epsilon>0$, it follows from H\"older's inequality that there exists $c>0$ with
\begin{align*}
&\mu(\{\phi:\lVert (\sqrt{-\Delta})^s\phi\rVert_{L_x^p(B_2\times\mathbb{T})}>\lambda\})\\
&\hspace{0.8in}\leq \mu(\{\phi:\lVert (\sqrt{-\Delta})^{s}\phi\rVert_{L_x^{4+\epsilon}}>c\lambda\})
\end{align*}
and
\begin{align*}
&\mu(\{\phi:\lVert S(t)\phi\rVert_{L_x^pL_t^q(I)}>\lambda\})\\
&\hspace{0.8in}\leq \mu(\{\phi:\lVert S(t)\phi\rVert_{L_x^{4+\epsilon}L_t^q(I)}>c\lambda\}).
\end{align*}

We therefore let $p>4$ be given, for which ($\ref{ref17}$) gives
\begin{align*}
\lVert e_n\rVert_{L_x^p(B_2)}\lesssim n^{\frac{1}{2}-\frac{2}{p}}.
\end{align*}
To establish ($\ref{ref32}$), suppose that $p<\rho_*(\alpha,s)$, and let $\lambda>0$ be given.  Now, fix $r\geq p$ and write
\begin{align*}
\mu(\{\phi:\lVert (\sqrt{-\Delta})^s\phi\rVert_{L^p}>\lambda\})&\leq \frac{1}{\lambda^r}\mathbb{E}\Big[ \lVert (\sqrt{-\Delta})^s\phi\rVert_{L^p}^r \Big]\\
&\leq \frac{1}{\lambda^r}\left\lVert \left(\mathbb{E}\Big[ |(\sqrt{-\Delta})^s\phi(x)|^{r}\Big]\right)^{1/r}\right\rVert_{L^p}^r\\
&\lesssim \frac{(\sqrt{r})^r}{\lambda^r}\left\lVert \bigg(\sum_{n,n'} z_{n,n'}^{2s}|\gamma_{n,n'}|^2|e_n(x_1,x_2)|^2\bigg)^{1/2} \right\rVert_{L^p}^r.
\end{align*}
This quantity is then equal to
\begin{align}
\nonumber &\frac{(\sqrt{r})^r}{\lambda^r}\bigg\lVert \sum_{n,n'} (z_{n,n'})^{2s}|\gamma_{n,n'}|^2|e_n(x_1,x_2)|^2\bigg\rVert_{L^{p/2}}^{r/2}\\
\nonumber &\hspace{0.2in}\leq\frac{(\sqrt{r})^r}{\lambda^r}\bigg(\sum_{n,n'} (z_{n,n'})^{2s}|\gamma_{n,n'}|^2\lVert e_n(x_1,x_2)\rVert_{L^{p}}^{2}\bigg)^{r/2}\\
&\hspace{0.2in}\lesssim\frac{(\sqrt{r})^r}{\lambda^r}\bigg(\sum_{n,n'} (z_{n,n'})^{2s}|\gamma_{n,n'}|^2n^{2(\frac{1}{2}-\frac{2}{p})}\bigg)^{r/2}\label{ref34}
\end{align}
where to obtain the second inequality we have used ($\ref{ref17}$) and recalled that $p>4$ holds by assumption.  Now, recalling the hypothesis $s<\alpha$, and using the eigenvalue bound $\lambda_n\gtrsim n$ (for $n$ sufficiently large, as a consequence of ($\ref{ref16}$)), we have
\begin{align}
\nonumber &\sum_{n,n'}(z_{n,n'})^{2s}|\gamma_{n,n'}|^2n^{2(\frac{1}{2}-\frac{2}{p})}\\
\nonumber &\hspace{0.2in}\lesssim \left(\sup_{n,n'} (z_{n,n'})^{2(s-\alpha)}n^{1-\frac{4}{p}}\right)\sum_{n,n'} (z_{n,n'})^{2\alpha}|\gamma_{n,n'}|^2\\
\nonumber &\hspace{0.2in}\lesssim \left(\sup_{n,n'}\, n^{1+2(s-\alpha)-\frac{4}{p}}\right)\sum_{n,n'} (z_{n,n'})^{2\alpha}|\gamma_{n,n'}|^2\\
&\hspace{0.2in}\lesssim \sup_{n,n'}\, n^{1-2(\alpha-s)-\frac{4}{p}},\label{ref35}
\end{align}  
where we have recalled that $\sum_{n,n'} (z_{n,n'})^{2\alpha}|\gamma_{n,n'}|^2<\infty$ since $(F,G)\in H_x^\alpha\times H_x^{\alpha-1}$.

Suppose first that $\alpha-s\geq \frac{1}{2}$ (recall that in this case $p$ is subject only to the restriction $p>4$).  We then have $1-2(\alpha-s)-\frac{4}{p}<0$, so that the right-hand side of ($\ref{ref35}$) is finite.  Alternatively, if $0<\alpha-s<\frac{1}{2}$, then we have $p\leq \frac{4}{1-2(\alpha-s)}$ (and thus $1-2(\alpha-s)-\frac{4}{p}<0$) so that the right-hand side of ($\ref{ref35}$) is again finite.  In view of ($\ref{ref34}$), in both cases we have thus obtained
\begin{align}
\mu(\{\phi:\lVert (\sqrt{-\Delta})^s\phi\rVert_{L^p}>\lambda\})&\lesssim (\sqrt{r}/\lambda)^r.\label{ref38}
\end{align}
Minimizing the right hand side of ($\ref{ref38}$) over all $r\geq p$, this completes the proof of ($\ref{ref32}$).

To complete the proof of the lemma, it remains to show ($\ref{ref33}$).  For this, we argue as above.  In particular, letting $4<p<\rho_*(\alpha)$, $2\leq q<\infty$ and $\lambda>0$ be given, we have
\begin{align}
\mu(\{\phi:\lVert S(t)\phi\rVert_{L_x^pL_t^q(I)}>\lambda\})&\leq \frac{(\sqrt{r})^r}{\lambda^r}\bigg\lVert\bigg(\sum_{n,n'} |\gamma_{n,n'}|^2|e_n(x_1,x_2)|^2\bigg)^{1/2}\bigg\rVert_{L_x^pL_t^q(I)}^r\label{ref39}
\end{align}
for every $r\geq \max\{p,q\}$.  Arguing as in ($\ref{ref34}$), we obtain the bound
\begin{align}
\nonumber(\ref{ref39})&\leq \frac{(\sqrt{r})^r}{\lambda^r}\bigg[\bigg(\sup_{n,n'}\, (z_{n,n'})^{-2\alpha}n^{1-\frac{4}{p}}\bigg)\sum_{n,n'} |\gamma_{n,n'}|^2\bigg]^{r/2},
\end{align}
and thus, arguing as above,
\begin{align*}
\mu(\{\phi:\lVert S(t)\phi\rVert_{L_x^pL_t^q(I)}>\lambda\})&\lesssim (\sqrt{r}/\lambda)^r.
\end{align*}
Optimizing in the choice of $r$ now gives ($\ref{ref33}$) as desired.  This completes the proof of the lemma.
\end{proof}

\section{Estimates of the nonlinearity}
\label{ref43}

In this section we establish two lemmas which will provide estimates for the nonlinear term of the Duhamel formula ($\ref{ref9}$).  These lemmas, when combined with the probabilistic bounds of Section $\ref{ref28}$, will facilitate the proof of the local well-posedness result stated in Theorem $\ref{ref10}$ by allowing us to close a contraction mapping argument in suitable $X^{s,b}$ norms.

We begin with the following lemma.
\begin{lemma}
\label{ref44}
Fix $0<s\leq 1$ and suppose that $p\in [1,\infty)$ satisfies 
\begin{align*}
\textrm{(i)}\,\, p>\frac{2}{2-s}\,\, \textrm{if}\,\, 0<s<\frac{1}{2}, \quad\quad \textrm{(ii)}\,\, p>\frac{6}{5-2s}\,\, \textrm{if}\,\, \frac{1}{2}\leq s<1, and
\end{align*}
\begin{align*}
\textrm{(iii)}\,\, p\geq 2\,\, \textrm{if}\,\, s=1.
\end{align*}

Then there exists $\epsilon>0$ such that for every $b\in (\frac{1}{2},\frac{1}{2}+\epsilon)$ there exists a constant $C_b>0$ such that, for all $f\in L_x^p(B_2\times \mathbb{T};L_t^2(I))$,
\begin{align*}
\left\lVert \int_0^t S(t-\tau)(\sqrt{-\Delta})^{-1}f(\tau)d\tau\right\rVert_{X^{s,b}(I)}\leq C_b\lVert f\rVert_{L_x^p(B_2\times \mathbb{T};L_t^2(I))},
\end{align*}
where $I=[0,T]$ with $0<T<1$.
\end{lemma}

\begin{proof}
Fix $\epsilon>0$ (to be chosen later in the argument) and let $b\in\mathbb{R}$ be given such that $\frac{1}{2}<b<\frac{1}{2}+\epsilon$.  Now, arguing as in the proof of the bound (21) in \cite[Lemma $2.5$]{BB2}, it suffices to estimate
\begin{align}
\left(\sum_{m,n,n'} \frac{|\hat{f}(m,n,n')|^2}{\langle z_{n,n'}\rangle^{2(1-s)}\langle 2\pi m+z_{n,n'}\rangle^{2(1-b)}}\right)^{1/2}.\label{ref44.5}
\end{align}

We bound ($\ref{ref44.5}$) by duality, writing
\begin{align*}
(\ref{ref44.5})&=\sup_{g\in \mathcal{G}} \int_I\int_{B_2\times\mathbb{T}} f(t,x)\overline{g(t,x)}dxdt
\end{align*}
where the set $\mathcal{G}$ consists of functions $g:I\times (B_2\times\mathbb{T})\rightarrow \mathbb{C}$ of the form
\begin{align*}
g(t,x)&=\sum_{m,n,n'} \frac{g_{m,n,n'}}{\langle z_{n,n'}\rangle^{1-s}\langle 2\pi m+z_{n,n'}\rangle^{1-b}}e_n(x_1,x_2)e^{2\pi i(n'x_3+mt)}
\end{align*}
which satisfy
\begin{align*}
\sum_{m,n,n'} |g_{m,n,n'}|^2\leq 1.
\end{align*}

For each $g\in \mathcal{G}$, note that by the H\"older inequality, one has
\begin{align*}
\int_I\int_{B_2\times\mathbb{T}} f\overline{g}&\leq \lVert f\rVert_{L_x^pL_t^2(I)}\lVert g\rVert_{L_x^qL_t^2(I)}
\end{align*}
with $\frac{1}{p}+\frac{1}{q}=1$. We now consider cases depending on the value of $s$.  Suppose first that $0<s<\frac{1}{2}$ holds.  We then get the inequalities
\begin{align}
\nonumber \lVert g\rVert_{L_x^qL_t^2(I)}&\leq \bigg\lVert\bigg(\sum_m\bigg(\sum_{n,n'} \frac{|g_{m,n,n'}|}{\langle z_{n,n'}\rangle^{1-s}\langle 2\pi m+z_{n,n'}\rangle^{1-b}}|e_n(x_1,x_2)|\bigg)^2\bigg)^{1/2}\bigg\rVert_{L_x^q}\\
\nonumber &\leq \bigg\lVert \sum_m \bigg(\sum_{n,n'} |g_{m,n,n'}|^2\bigg)\bigg(\sum_{n,n'} \frac{|e_n(x_1,x_2)|^2}{\langle z_{n,n'}\rangle^{2(1-s)}\langle 2\pi m+z_{n,n'}\rangle^{2(1-b)}}\bigg)\bigg\rVert_{L_x^{q/2}}^{1/2}\\
&=\bigg\lVert \sum_{n,n'} \frac{|e_n(x_1,x_2)|^2}{\langle z_{n,n'}\rangle^{1-2(s+\epsilon)} }\beta_{n,n'}\bigg\rVert_{L_x^{q/2}}^{1/2}\label{ref45}
\end{align}
with
\begin{align*}
\beta_{n,n'}:=\sum_m \bigg(\sum_{\widetilde{n},\widetilde{n}'} |g_{m,\widetilde{n},\widetilde{n}'}|^2\bigg) \langle z_{n,n'}\rangle^{-1-2\epsilon}\langle 2\pi m+z_{n,n'}\rangle^{-2(1-b)}.
\end{align*}

It now follows that we have the bound
\begin{align*}
(\ref{ref45})&\lesssim \bigg(\sum_{n,n'}\frac{\lVert e_n(x_1,x_2)\rVert_{L_x^{q}}^2}{\langle z_{n,n'}\rangle^{1-2(s+\epsilon)}}\beta_{n,n'}\bigg)^{1/2}\\
&\lesssim \bigg(\sup_{n,n'}\frac{\lVert e_n(x_1,x_2)\rVert_{L_x^{q}}}{\langle z_{n,n'}\rangle^{\frac{1}{2}-(s+\epsilon)}}\bigg)\bigg(\sum_{n,n'} \beta_{n,n'}\bigg)^{1/2}\\
&\lesssim \sup_{n,n'} \frac{\lVert e_n(x_1,x_2)\rVert_{L_x^{q}}}{\langle z_{n,n'}\rangle^{\frac{1}{2}-(s+\epsilon)}}.
\end{align*}
In view of the eigenfunction estimates (\ref{ref17}), this quantity is finite for all $q<\frac{2}{s}$ (this corresponds to the restriction $p>\frac{2}{2-s}$).  This completes the proof when $0<s<\frac{1}{2}$.

Suppose now that $\frac{1}{2}\leq s<1$.  By Minkowski's inequality for integrals and the  Sobolev embedding, we then get (for $q\geq 2$)
\begin{align*}
\lVert g\rVert_{L_x^qL_t^2(I)}\leq \lVert g\rVert_{L_t^2(I;L_x^q)}\lesssim \lVert g\rVert_{L_t^2(I;H_x^{\frac{3}{2}-\frac{3}{q}})}
\end{align*}
which is bounded by $\lVert g\rVert_{X^{1-s,1-b}(I)}\leq 1$ when $\frac{3}{2}-\frac{3}{q}<1-s$ (which corresponds to $p>\frac{6}{5-2s}$).

It remains to consider the case $s=1$, where the desired conclusion follows by taking $q=2$ (and thus $p=2$) to obtain
\begin{align*}
\lVert g\rVert_{L_{t,x}^2}&\leq \lVert g\rVert_{X^{1-s,1-b}(I)}\leq 1.
\end{align*}

This completes the proof of Lemma $\ref{ref44}$.
\end{proof}

The second lemma of this section expresses an embedding of the form 
\begin{align*}
X^{s,b}(I)\hookrightarrow L_x^pL_t^q(I)
\end{align*}
for suitable values of $s>0$ and $2\leq p,q<\infty$.  

\begin{lemma}
\label{ref46}
Fix $s>0$, $b>\frac{1}{2}$, and $2\leq p,q<\infty$.  Suppose further that $s>\max\{1-\frac{1}{q},\frac{3}{2}-\frac{1}{q}-\frac{2}{p}\}$.  Then there exists $C>0$ and $\delta>0$ such that
\begin{align}
\lVert f\rVert_{L_x^pL_t^q(I)}\leq C\lVert f\rVert_{X^{s,b}(I)}\label{ref47}
\end{align}
for all $f\in X^{s,b}(I)$.
\end{lemma}

\begin{proof}
Suppose $f\in X^{s,b}(I)$ satisfies $\lVert f\rVert_{X^{s,b}(I)}\leq 1$.  We may then write
\begin{align}
\lVert f\rVert_{L_x^{p}L_t^{q}(I)}&=\bigg\lVert \sum_{m,n,n'} \frac{a_{m,n,n'}}{\langle z_{n,n'}\rangle^{s}\langle 2\pi m+z_{n,n'}\rangle^{b}}e_n(x_1,x_2)e^{2\pi i(n'x_3+mt)}\bigg\rVert_{L_x^pL_t^q(I)}\label{ref48}
\end{align}
where $(a_{m,n,n'})$ is a sequence of complex numbers with $\sum_{m,n,n'}|a_{m,n,n'}|^2\leq 1$.  We then obtain
\begin{align*}
(\ref{ref48})&\lesssim \bigg\lVert \sum_{\ell\in \mathbb{Z}}\frac{1}{\langle \ell\rangle^b}\bigg|\sum_{\substack{m,n,n'\\m+\lfloor z_{n,n'}/(2\pi)\rfloor=\ell}}\frac{a_{m,n,n'}}{\langle z_{n,n'}\rangle^{s}}e_n(x_1,x_2)e^{i(2\pi n'x_3+2\pi mt)}\bigg|\,\bigg\rVert_{L_x^pL_t^q(I)}\\
&=\bigg\lVert \sum_{\ell\in\mathbb{Z}}\frac{b_{\ell}}{\langle \ell\rangle^b}\bigg|\sum_{m,n,n'} \frac{c_{\ell,m,n,n'}}{\langle z_{n,n'}\rangle^{s}}e_n(x_1,x_2)e^{2\pi in'x_3}e^{2\pi imt}\bigg|\,\bigg\rVert_{L_x^pL_t^q(I)}
\end{align*}
where we have set 
\begin{align*}
b_{\ell}:=\bigg(\sum_{\substack{m,n,n'\\m+\lfloor z_{n,n'}/(2\pi)\rfloor=\ell}} |a_{m,n,n'}|^2\bigg)^{1/2},\quad c_{\ell,m,n,n'}(t):=\frac{a_{m,n,n'}}{b_{\ell}}.
\end{align*}

Invoking Minkowski's inequality, we obtain
\begin{align}
\nonumber(\ref{ref48})&\leq \sum_{\ell\in\mathbb{Z}}\frac{b_{\ell}}{\langle \ell\rangle^b}\bigg\lVert \sum_{\substack{m,n,n'\\m+\lfloor z_{n,n'}/(2\pi)\rfloor=\ell}} \frac{c_{\ell,m,n,n'}}{\langle z_{n,n'}\rangle^{s}}e_n(x_1,x_2)e^{2\pi i n'x_3}e^{2\pi imt}\bigg\rVert_{L_x^pL_t^q(I)}\\
&\lesssim \sup_{\ell\in\mathbb{Z}} \bigg\lVert \sum_{m,n,n'} \frac{c_{\ell,m,n,n'}}{\langle z_{n,n'}\rangle^{s}}e_n(x_1,x_2)e^{2\pi i n'x_3}e^{2\pi imt}\bigg\rVert_{L_x^pL_t^q(I)}\label{ref49}
\end{align}
where to obtain the last inequality we have observed that, in view of the Cauchy-Schwarz inequality and the conditions $b>\frac{1}{2}$, $\sum_{m,n,n'} |a_{m,n,n'}|^2\leq 1$, one has the bounds
\begin{align*}
\sum_{\ell\in\mathbb{Z}} \frac{b_{\ell}}{\langle \ell\rangle^b}&\leq \bigg(\sum_{\ell} \frac{1}{\langle \ell\rangle^{2b}}\bigg)^{1/2}\bigg(\sum_{\ell} b_{\ell}^2\bigg)^{1/2}
\end{align*}
so that
\begin{align*}
\sum_{\ell\in\mathbb{Z}} \frac{b_{\ell}}{\langle \ell\rangle^b}&\lesssim 1.
\end{align*}

To estimate the quantity on the right-hand side of (\ref{ref49}), fix $\ell\in\mathbb{Z}$ and use Lemma $\ref{ref23}$ to obtain the bound
\begin{align}
\nonumber &\bigg\lVert \sum_{\substack{m,n,n'\\m+\lfloor z_{n,n'}/(2\pi)\rfloor=\ell}} \frac{c_{\ell,m,n,n'}}{\langle z_{n,n'}\rangle^{s}}e_n(x_1,x_2)e^{2\pi i n'x_3}e^{2\pi imt}\bigg\rVert_{L_x^pL_t^q(I)}\\
\nonumber &\hspace{0.2in}=\bigg\lVert \sum_{\substack{m,n,n'\\m+\lfloor z_{n,n'}/(2\pi)\rfloor=\ell}} \frac{c_{\ell,m,n,n'}}{\langle z_{n,n'}\rangle^{s}}e_n(x_1,x_2)e^{2\pi i n'x_3}e^{-2\pi i\lfloor z_{n,n'}/(2\pi)\rfloor t}e^{2\pi i\ell t}\bigg\rVert_{L_x^pL_t^q(I)}\\
&\hspace{0.2in}\lesssim \bigg\lVert \sum_{m,n,n'} \frac{c_{\ell,m,n,n'}}{\langle z_{n,n'}\rangle^{s}}e_n(x_1,x_2)e^{2\pi i n'x_3}\bigg\rVert_{H_x^{\sigma}},\label{ref50}
\end{align}
for all $\sigma>0$ with $\sigma>\max\{1-\frac{1}{q},\frac{3}{2}-\frac{1}{q}-\frac{2}{p}\}$.

The right-hand side of ($\ref{ref50}$) is equal to 
\begin{align*}
\bigg(\sum_{\substack{m,n,n'\\m+\lfloor z_{n,n'}/(2\pi)\rfloor=\ell}} \frac{|c_{\ell,m,n,n'}|^2}{\langle z_{n,n'}\rangle^{2(s-\sigma)}}\bigg)^{1/2},
\end{align*}
so that choosing $\sigma>0$ such that 
\begin{align*}
\max\bigg\{1-\frac{1}{q},\frac{3}{2}-\frac{1}{q}-\frac{2}{p}\bigg\}<\sigma\leq s
\end{align*}
we obtain
\begin{align*}
\lVert f\rVert_{L_x^pL_t^q(I)}\leq C\sup_{\ell\in \mathbb{Z}}\bigg(\sum_{\substack{m,n,n'\\m+\lfloor z_{n,n'}/(2\pi)\rfloor=\ell}} |c_{\ell,m,n,n'}|^2\bigg)^{1/2}.
\end{align*}

The desired inequality (\ref{ref47}) now follows by observing that, for each $\ell\in\mathbb{Z}$, the inequality $$\sum_{\substack{m,n,n'\\m+\lfloor z_{n,n'}/(2\pi)\rfloor=\ell}} |c_{\ell,m,n,n'}|^2\leq 1$$ holds by construction.
\end{proof}

\section{Local well-posedness: contraction in $X^{s,b}$ spaces}
\label{ref51}

In this section we give the proof of Theorem $\ref{ref10}$, the local well-posedness result.  As we described in the introduction, the proof is based on a contraction-mapping argument in the spaces $X^{s,b}$, using the results of Sections $2$, $3$ and $4$.  To simplify notation, in this section and in the remainder of the paper we shall make use of the abbreviation
\begin{align}
\mathcal{F}(z)=|\iRe(z)|^\gamma(\iRe(z)), \,\, z\in\mathbb{C},\label{ref52}
\end{align}
when convenient.

\begin{proof}[Proof of Theorem $\ref{ref10}$]
Fix $0<s\leq 1$ and $b>\frac{1}{2}$.  We proceed by a fixed point argument.  In particular, for each $A\geq 1$ we shall identify positive constants $T$ and $R$ (each depending on $A$) and construct a set $\Sigma_A$ satisfying the conditions stated above such that for all $\omega\in \Omega\setminus\Sigma_A$ the map $\Phi_\omega:B_R\rightarrow B_R$ (where $B_R:=\{v\in X^{s,b}([0,T]):\lVert v\rVert_{X^{s,b}}\leq R\}$) given by 
\begin{align*}
&[\Phi_{\omega}(v)](t,x):=\\
&\hspace{0.4in}-i\lambda\int_0^t S(t-t')(\sqrt{-\Delta})^{-1}\Big[|\iRe(S(t')\phi_{\omega}+v(t'))|^\gamma \iRe(S(t')\phi_\omega+v(t'))\Big]dt'
\end{align*}
is a contraction.  In what follows, we will use the abbreviated notation $\Phi=\Phi_\omega$ and $\phi=\phi_\omega$ when there is no potential for ambiguity.

Let $A\geq 1$ be given.  Fix $R>0$ and $0<T<1$ to be determined later in the argument, and set $I=[0,T]$.  Let $v\in X^{s,b}(I)$ be given with $\lVert v\rVert_{X^{s,b}}\leq R$.  Now, fix $p\geq 1$ and $\epsilon>0$ satisfying the conditions of Lemma $\ref{ref44}$.  Invoking that lemma, we get
\begin{align}
\nonumber \lVert \Phi(v)\rVert_{X^{s,b}(I)}&\lesssim \lVert S(t)\phi+v(t)\rVert_{L_x^{p(\gamma+1)}(B_2\times \mathbb{T};L_t^{2(\gamma+1)}(I))}^{\gamma+1}\\
\nonumber &\lesssim \lVert S(t)\phi\rVert_{L_x^{p(\gamma+1)}(B_2\times \mathbb{T};L_t^{2(\gamma+1)}(I))}^{\gamma+1}\\
&\hspace{1.2in}+\lVert v\rVert_{L_x^{p(\gamma+1)}(B_2\times \mathbb{T};L_t^{2(\gamma+1)}(I))}^{\gamma+1}\label{ref53}
\end{align}
under the condition $b<\frac{1}{2}+\epsilon$.

To estimate the linear evolution $S(t)\phi$, we appeal to the probabilistic considerations of Section $3$.  In particular, if $p$ is chosen such that 
\begin{align}
2\leq p(\gamma+1)<\rho_*(\alpha)\label{ref54}
\end{align}
(where $\rho_*(\alpha)$ is as in ($\ref{ref30}$)), then an application of Lemma \ref{ref31} gives the bound
\begin{align}
\lVert S(t)\phi\rVert_{L_x^{p(\gamma+1)}L_t^{2(\gamma+1)}(I)}\leq A\label{ref55}
\end{align}
for all $\omega\in \Omega\setminus \Sigma_A$, with
\begin{align*}
\Sigma_A:=\{\omega\in \Omega:\lVert S(t)\phi_\omega\rVert_{L_x^{p(\gamma+1)}L_t^{2(\gamma+1)}(I)}>A\},
\end{align*}
\begin{align*}
\mu(\Sigma_A)\lesssim \exp(-cA^2).
\end{align*}

On the other hand, to estimate the $L_x^{p(\gamma+1)}L_t^{2(\gamma+1)}(I)$ norm of $v$, we note that Lemma $\ref{ref46}$ gives
\begin{align}
\lVert v\rVert_{L_x^{p(\gamma+1)}L_t^{2(\gamma+1)}(I)}&\leq CT^{1/\epsilon}\lVert v\rVert_{X^{s,b}},\label{ref56}
\end{align}
for some $\epsilon>0$, provided that the parameters $s$ and $p$ satisfy the condition
\begin{align}
s>\max\bigg\{1-\frac{1}{2(\gamma+1)},\frac{3}{2}-\frac{1}{2(\gamma+1)}-\frac{2}{p(\gamma+1)}\bigg\}.\label{ref57}
\end{align}
Indeed, assuming ($\ref{ref57}$), the bound ($\ref{ref56}$) is obtained by choosing $q\geq 2(\gamma+1)$ such that
\begin{align}
s>\max\bigg\{1-\frac{1}{q},\frac{3}{2}-\frac{1}{q}-\frac{2}{p(\gamma+1)}\bigg\},\label{ref58}
\end{align}
and using the H\"older inequality followed by Lemma $\ref{ref46}$ with the norm $L_x^{p(\gamma+1)}L_t^q(I)$ to obtain
\begin{align*}
\lVert v\rVert_{L_x^{p(\gamma+1)}L_t^{2(\gamma+1)}(I)}&\lesssim T^{\frac{1}{2(\gamma+1)}-\frac{1}{q}}\lVert v\rVert_{L_x^{p(\gamma+1)}L_t^q(I)}\\
&\lesssim T^{\frac{1}{2(\gamma+1)}-\frac{1}{q}}\lVert v\rVert_{X^{s,b}},
\end{align*}
Taking $q$ as large as possible while satisfying ($\ref{ref58}$), we obtain the desired estimate ($\ref{ref56}$) for suitable choice of $\epsilon>0$.

Combining ($\ref{ref53}$) with ($\ref{ref55}$) and ($\ref{ref56}$), we have
\begin{align*}
\lVert \Phi_{\omega}(v)\rVert_{X^{s,b}(I)}\leq C_1A^{\gamma+1}+C_2(T^{1/\epsilon}R)^{\gamma+1}.
\end{align*}
for all $\omega\in \Omega\setminus \Sigma_A$, provided that there exist $s>0$ and $p\geq 1$ satisfying the condition of Lemma $\ref{ref44}$ together with ($\ref{ref54}$) and ($\ref{ref57}$).  Note that the condition ($\ref{thm1cond1}$)--($\ref{thm1cond2}$) implies that such values of $s$ and $p$ exist (see Appendix $\ref{ref91}$ for further comments on this point).  
Choosing $R=2\max\{C_1,C_2\}A^{\gamma+1}$ and 
\begin{align*}
T<\Big(\frac{1}{2C_2R^\gamma}\Big)^\frac{\epsilon}{\gamma+1}
\end{align*}
now gives
\begin{align*}
\lVert \Phi_{\omega}(v)\rVert_{X^{s,b}(I)}\leq R
\end{align*}
for all $\omega\in \Omega\setminus \Sigma_A$.  This shows that for all such $\omega$, the map $\Phi_\omega$ carries the set $B_R$ to itself.

To conclude the desired existence and uniqueness result, it remains to show that $\Phi_\omega$ is a contraction on $B_R$ whenever $\omega$ is not in $\Sigma_A$.  This follows almost immediately from the above arguments and the elementary inequality
\begin{align}
|\mathcal{F}(a)-\mathcal{F}(b)|\lesssim |a-b|(|a|^\gamma+|b|^\gamma).\label{ref60}
\end{align}
We include the details for the convenience of the reader: let $v,w\in B_R$ be given and suppose that $s$ and $p$ are as above.  Arguing as before, we obtain
\begin{align}
\nonumber &\lVert \Phi_\omega(v)-\Phi_\omega(w)\rVert_{X^{s,b}}\\
\nonumber &\hspace{0.4in}\lesssim \lVert \mathcal{F}(S(t)\phi+v(t))-\mathcal{F}(S(t)\phi+w(t))\rVert_{L_x^{p}L_t^{2}(I)}\\
\nonumber &\hspace{0.4in}\lesssim \lVert |v(t)-w(t)|(|S(t)\phi+v(t)|^\gamma+|S(t)\phi+w(t)|^\gamma)\rVert_{L_x^{p}L_t^2(I)}\\
\nonumber &\hspace{0.4in}\leq \lVert v-w\rVert_{L_x^{p(\gamma+1)}L_t^{2(\gamma+1)}(I)}\\
&\hspace{0.8in} \cdot \lVert |S(t)\phi+v(t)|^\gamma+|S(t)\phi+w(t)|^\gamma\rVert_{L_x^{\frac{p(\gamma+1)}{\gamma}}L_t^{\frac{2(\gamma+1)}{\gamma}}(I)},\label{ref61}
\end{align}
where the third line follows from ($\ref{ref60}$) and the fourth line results from an application of H\"older's inequality.  Then, using H\"older's inequality (in time) and Lemma $\ref{ref46}$ to estimate the norm of $v-w$, we get
\begin{align*}
(\ref{ref61})&\lesssim T^{1/\epsilon}\lVert v-w\rVert_{X^{s,b}}\\
&\hspace{0.4in}\cdot (\lVert S(t)\phi\rVert_{L_x^{p(\gamma+1)}L_t^{2(\gamma+1)}(I)}^\gamma+\lVert v\rVert_{L_x^{p(\gamma+1)}L_t^{2(\gamma+1)}(I)}^\gamma+\lVert w\rVert_{L_x^{p(\gamma+1)}L_t^{2(\gamma+1)}(I)}^\gamma)\\
&\lesssim T^{(\gamma+1)/\epsilon}\lVert v-w\rVert_{X^{s,b}}(\lVert S(t)\phi\rVert_{L_x^{p(\gamma+1)}L_t^{q}(I)}^\gamma+\lVert v\rVert_{X^{s,b}}^\gamma+\lVert w\rVert_{X^{s,b}}^\gamma) 
\end{align*}
with $q>2(\gamma+1)$ chosen as before (see ($\ref{ref58}$) above).

Recalling that $v$ and $w$ each belong to $B_R$, this gives
\begin{align*}
\lVert \Phi_\omega(v)-\Phi_\omega(w)\rVert_{X^{s,b}}\leq CT^{(\gamma+1)/\epsilon}(A^\gamma+2R^\gamma)\lVert v-w\rVert_{X^{s,b}}
\end{align*}
for all $\omega\in \Omega\setminus \Sigma_A$.  Recalling the choice of $R$ above as a multiple of a power of $A$, and choosing 
\begin{align*}
T<\min\bigg\{\Big(\frac{1}{2C_2R^\gamma}\Big)^{\frac{\epsilon}{\gamma+1}},\Big(\frac{1}{2C(A^\gamma+2R^{\gamma})}\Big)^{\frac{\epsilon}{\gamma+1}}\bigg\},
\end{align*}
we obtain the desired contraction property for all such $\omega$.  

It follows immediately that there is a unique $v\in B_R$ such that $\Phi(v)=v$.  We conclude the proof of the proposition by noting that the function $u(t)=S(t)\phi_\omega+v(t)$ is the desired solution of the initial value problem.
\end{proof}

\begin{remark}
Under the hypotheses of Theorem $\ref{ref10}$, the constructed solution satisfies the bound
\begin{align*}
\lVert u\rVert_{L_t^\infty(I;H_x^{s'})}\leq CA^{\gamma+2},\quad 0<s'<\min\{s,\alpha-1\}.
\end{align*}
with $s$ as in the statement of the theorem.  Indeed, this is an immediate consequence of ($\ref{ref12}$), when combined with the embedding $X^{s,b}\hookrightarrow C_t(I;H_x^{s})$ (see ($\ref{ref19}$) in Section $2$) and the conservation of $H^s$ norms for the linear propagator.
\conclremark\end{remark}

\appendix 

\section{The choice of $s$ and $p$ in the proof of Theorem $\ref{ref10}$}
\label{ref91}

In this appendix, we show how the condition ($\ref{thm1cond1}$)--($\ref{thm1cond2}$) allows for the choice of the parameters $s>0$ and $p\geq 2$ used in the proof of Theorem $\ref{ref10}$.  Recall that the initial value problem (NLW) under consideration in this paper is equipped with nonlinearity $|u|^\gamma u$; for the reformulated equation ($\ref{ref5}$), the nonlinearity takes the form $\mathcal{F}(u)=|\iRe(u)|^\gamma\iRe(u)$.

We shall choose the values of $s$ and $p$ to satisfy the condition of the nonlinear $X^{s,b}$ estimate in Lemma $\ref{ref44}$ together with the conditions ($\ref{ref54}$) and ($\ref{ref57}$), which correspond respectively to applications of the probabilistic estimate of Lemma $\ref{ref31}$ and the $X^{s,b}(I)\hookrightarrow L_x^pL_t^q(I)$ embedding of Lemma $\ref{ref46}$.  We first note that the condition $$s>1-\frac{1}{2(\gamma+1)}$$ imposed as part of ($\ref{ref57}$) implies that for $\gamma>0$, we must have $s>1/2$.

Let $0<\gamma<2$ and $\alpha>0$ be given.  If $\alpha\geq 1/2$, the condition ($\ref{ref54}$) reads only $p\geq 2/(\gamma+1)$ (which is automatically satisfied for $p\geq 2$).  Then, choosing $s=1$, the condition $s>1-\frac{1}{2(\gamma+1)}$ in ($\ref{ref57}$) is automatically satisfied.  The hypothesis of Lemma $\ref{ref44}$ in this case becomes $p\geq 2$.  It remains to determine when we can satisfy the remaining condition in ($\ref{ref57}$), which reads $$1>3/2-1/(2(\gamma+1))-2/(p(\gamma+1)),$$ and which is satisfied whenever $p < 4/\gamma$.  Collecting these criteria, we must choose $p\in [2,4/\gamma)$ which is possible for $\gamma<2$.  Direct examination of the criteria that arise when $1/2\leq s<1$ shows that reducing the value of $s$ does not lead to a larger admissible range of $\gamma$ in this case.\footnote{Indeed, in the $\alpha\geq 1/2$ case, the subcubic problem is no longer supercritical with respect to scaling, and one can appeal to a stronger class of deterministic fixed-point estimates; since our focus in this paper is on the probabilistic theory, we do not pursue this issue further.}

We now examine the regime $\alpha<1/2$, where ($\ref{ref54}$) becomes $\frac{2}{\gamma+1}\leq p<\frac{4}{(\gamma+1)(1-2\alpha)}$.  Here, we first look at situations where $s=1$.  In this case, the hypothesis of Lemma $\ref{ref44}$ is again $p\geq 2$, and the first condition in ($\ref{ref57}$) is again automatically satisfied.  The second condition in ($\ref{ref57}$) again becomes $p<4/\gamma$, so that the conditions on $p$ in this regime are $$2\leq p\leq \min\bigg\{\frac{4}{\gamma},\frac{4}{(\gamma+1)(1-2\alpha)}\bigg\},$$ which is a nonempty interval when $$0<\gamma<\min\{2,\frac{2\alpha+1}{1-2\alpha}\}.$$ Fixing $\gamma\in (0,2)$, it follows that choice of the pair $(s,p)$ is possible (with $s=1$) when $\alpha>\frac{\gamma-1}{2(\gamma+1)}$.

To explore handling smaller values of $\alpha$ in this regime, we recall from above that ($\ref{ref57}$) implies $s>1/2$, and consider choices of $s$ in the range $1/2\leq s<1$, for which the hypothesis of Lemma $\ref{ref44}$ becomes $p>\frac{6}{5-2s}$ (i.e. $s<\frac{5}{2}-\frac{3}{p}$).  We consider two possibilities, (i) when $\frac{1}{2}-\frac{2}{p(\gamma+1)}\geq 0$, and (ii) when $\frac{1}{2}-\frac{2}{p(\gamma+1)}<0$, corresponding to situations where the condition ($\ref{ref57}$) becomes $$s>\frac{3}{2}-\frac{1}{2(\gamma+1)}-\frac{2}{p(\gamma+1)}=\Big( 1 - \frac{1}{2(\gamma+1)} \Big) + \Big( \frac{1}{2}-\frac{2}{p(\gamma+1)} \Big)$$ (when (i) holds) or 
\begin{align}
s>1-\frac{1}{2(\gamma+1)}\label{ref57a}
\end{align}
(when (ii) holds).  

If (i) holds, i.e.  $p\geq \frac{4}{\gamma+1}$, the range of admissible values for $s$ becomes $\max\bigg\{\frac{1}{2}, \frac{3}{2}-\frac{1}{2(\gamma+1)}-\frac{2}{p(\gamma+1)}\bigg\}<s<\min\bigg\{1,\frac{5}{2}-\frac{3}{p}\bigg\}$, which is nonempty provided that $$p>\max\bigg\{\frac{3}{2},\frac{2(3\gamma+1)}{2\gamma+3}\bigg\}\quad\textrm{and}\quad p<\frac{4}{\gamma}.$$ Recalling the additional restriction $p<\frac{4}{(\gamma+1)(1-2\alpha)}$, the conditions on $p$ read $$\max\bigg\{\frac{4}{\gamma+1},\frac{3}{2},\frac{2(3\gamma+1)}{2\gamma+3}\bigg\}<p<\min\bigg\{\frac{4}{\gamma},\frac{4}{(\gamma+1)(1-2\alpha)}\bigg\},$$ which is nonempty provided $$\gamma<\min\bigg\{2,\frac{5+6\alpha}{3(1-2\alpha)},\frac{4\alpha+\sqrt{4\alpha^2-24\alpha+15}}{3(1-2\alpha)}\bigg\}.$$  Note that all such situations are already handled by the $s=1$ case.

On the other hand, when (ii) holds we do obtain a result expanding the range of $\alpha$ whenever $0<\gamma<\sqrt{\frac{5}{3}}$.  Indeed, in case (ii), we have $p<\frac{4}{\gamma+1}$, and ($\ref{ref57}$) becomes ($\ref{ref57a}$).  The range of admissible $s$ then becomes $1-\frac{1}{2(\gamma+1)}<s<\min\bigg\{1,\frac{5}{2}-\frac{3}{p}\bigg\}$, which is nonempty provided $p>\frac{6(\gamma+1)}{3\gamma+4}$.  Recalling from ($\ref{ref54}$) the additional restriction $\frac{2}{\gamma+1}\leq p<\frac{4}{(\gamma+1)(1-2\alpha)}$, the conditions on $p$ in this case read $$\max\bigg\{\frac{2}{\gamma+1}, \frac{6(\gamma+1)}{3\gamma+4} \bigg\} < p < \min\bigg\{ \frac{4}{\gamma+1}, \frac{4}{(\gamma+1)(1-2\alpha)} \bigg\}$$ which is a nonempty interval whenever $$ 0<\gamma<\sqrt{5/3}.$$ It follows that choice of $(s,p)$ satisfying the deesired conditions is possible whenever $\gamma$ and $\alpha$ satisfy ($\ref{thm1cond1}$)--($\ref{thm1cond2}$), as desired.

\section*{Acknowledgments}

The author would like to express sincere thanks to the referee for many valuable comments and suggestions.


\begin{thebibliography}{9}
\bibitem{BenyiOhPocovnicu2019} \'A. B\'enyi, T. Oh, and O. Pocovnicu. On the probabilistic Cauchy theory for nonlinear dispersive PDEs. Landscapes of time-frequency analysis, 1–32, Appl. Numer. Harmon. Anal., Birkh\"auser/Springer, Cham (2019).
\bibitem{B1} J. Bourgain. Periodic nonlinear Schr\"odinger equation and invariant measures.  Comm. Math. Phys. 166 (1994), 1--24.
\bibitem{B2} J. Bourgain. Invariant measures for the 2D-Defocusing Nonlinear Schr\"odinger Equation.  Comm. Math. Phys. 176 (1996), 421--445.
\bibitem{B3} J. Bourgain. Invariant measures for the Gross-Piatevskii equation.  J. Math. Pures Appl. (9) 76 (1997), no. 8, 649--702.
\bibitem{BB1} J. Bourgain and A. Bulut.  Gibbs measure evolution in radial nonlinear wave and Schrödinger equations on the ball.  C. R. Mathematique 350 (2012) 11-12, p. 571--575.
\bibitem{BB2} J. Bourgain and A. Bulut.  Almost sure global well posedness for the radial nonlinear Schrodinger equation on the unit ball I: the 2D case.  Ann. Inst. H. Poincare C Anal. Non Lin\'eaire 31 (2014), no. 6, 1267--1288.
\bibitem{BB3} J. Bourgain and A. Bulut.  Almost sure global well posedness for the radial nonlinear Schrodinger equation on the unit ball II: the 3D case.  J. Eur. Math. Soc. 16 (2014), no. 6, 1289--1325.
\bibitem{BB4} J. Bourgain and A. Bulut.  Invariant Gibbs measure evolution for the radial nonlinear wave equation on the 3d ball.  J. Funct. Anal. 266 (2014), no. 4, 2319--2340.
\bibitem{Bringmann} B. Bringmann.  Invariant Gibbs measures for the three-dimensional wave equation with a Hartree nonlinearity II: Dynamics.  Preprint (2022), arXiv:2009.04616
\bibitem{BringmannDengNahmodYue} B. Bringmann, Y. Deng, A. Nahmod, and H. Yue.  Invariant Gibbs measures for the three dimensional cubic nonlinear wave equation.  Preprint (2022), arXiv:2205.03893.
\bibitem{B-icm} N. Burq.  Random data Cauchy theory for dispersive partial differential equations.  Proceedings of the International Congress of Mathematicians.  Volume III, 1862--1883.  Hindustan Book Agency, New Delhi, 2010.
\bibitem{BLP} N. Burq, G. Lebeau and F. Planchon.  Global existence for energy critical waves in 3-D domains.  J. Amer. Math. Soc. 21 (2008), no. 3, 831--845.
\bibitem{BT1} N. Burq and N. Tzvektov.  Random data Cauchy theory for supercritical wave equations. I. Local theory.  Invent. Math. 173 (2008), no. 3, 449--475.
\bibitem{BT2} N. Burq and N. Tzvektov.  Random data Cauchy theory for supercritical wave equations. II. A global existence result.  Invent. Math. 173 (2008), no. 3, 447--496.
\bibitem{BT-pwp} N. Burq and N. Tzvektov.  Probabilistic well-posedness for the cubic wave equation. (2011) J. Eur. Math. Soc. 16 (2014), no. 1, 1--30.
\bibitem{CollianderOh} J. Colliander and T. Oh.  Almost sure well-posedness of the cubic nonlinear Schr\"odinger equation below $L^2(\mathbb{T})$.  Duke Math. J. 161 (2012), no. 3, 367--544.
\bibitem{LRS} J. Lebowitz, H. Rose and E. Speer.  Statistical mechanics of the nonlinear Schr\"odinger equation.  J. Statis. Phys. 50 (1988), no. 3--4, 657--687.
\bibitem{LM} J. L\"uhrmann and D. Mendelson.  Random data Cauchy theory for nonlinear wave equations of power-type on $\mathbb{R}^3$.  Comm. Par. Diff. Eq. 39 (2014), no. 12, 2262--2283.
\bibitem{LM2} J. L\"uhrmann and D. Mendelson.  On the almost sure global well-posedness of energy sub-critical nonlinear wave equations on $\mathbb{R}^3$.  New York J. Math. 22 (2016), 209--227.
\bibitem{NPS} A. Nahmod, N. Pavlovic and G. Staffilani.  Almost sure existence of global weak solutions for supercritical Navier-Stokes equations.  SIAM J. Math. Anal. 45 (2013), no. 6, 3431--3452.
\bibitem{OP} T. Oh and O. Pocovnicu.  Probabilistic global well-posedness of the energy-critical defocusing quintic nonlinear wave equation on $\mathbb{R}^3$. J. Math. Pures Appl. 105 (2016), 342--366.
\bibitem{OPN} T. Oh, O. Pocovnicu, and N. Tzvetkov.  Probabilistic local Cauchy theory of the cubic nonlinear wave equation in negative Sobolev spaces.  Ann. Inst. Fourier (Grenoble) 72 (2022), no.2, 771--830.
\bibitem{OhRobertTzvetkov} T. Oh, T. Robert, and N. Tzvetkov.  Stochastic nonlinear wave dynamics on compact surfaces.  Ann. H. Lebesgue 6 (2023), 161--223.
\bibitem{P} O. Pocovnicu.  Almost sure global well-posedness for the energy-critical defocusing nonlinear wave equation on $\mathbb{R}^d$, $d=4$ and $5$.  J. Eur. Math. Soc. 19 (2017), no. 8, 2521--2575.
\bibitem{Rudnik} Z. Rundik.  Personal communication.
\bibitem{SS} J. Shatah and M. Struwe.  Geometric wave equations.  Courant Lecture Notes in Mathematics, 2.  New York University, Courant Institute of Mathematical Sciences, New York; American Mathematical Society, Providence, RI. 1998.
\bibitem{SX} C. Sun and B. Xia.  Probabilistic well-posedness for supercritical wave equations with periodic boundary condition on dimension three. Illinois J. Math. 60 (2016), no. 2, 481--503.
\bibitem{T-disc} N. Tzvetkov.  Invariant measures for the nonlinear Schr\"odinger equation on the disc.  Dyn. Partial Diff. Eq. 3 (2006), no. 2, 111--160.
\bibitem{T-inv} N. Tzvetkov.  Invariant measures for the defocusing nonlinear Schr\"odinger equation.  Ann. Inst. Fourier 58 (2008), no. 7, 2543--2604.
\end{thebibliography}
\end{document}